\documentclass[conference]{IEEEtran}%
\usepackage{amsfonts}
\usepackage{amsmath}
\usepackage{amssymb}
\usepackage{graphicx}
\usepackage{cite}
\usepackage{pgf}
\usepackage{xcolor}
\usepackage{units}%
\newtheorem{theorem}{Theorem}

\newtheorem{conjecture}[theorem]{Conjecture}

\newtheorem{definition}[theorem]{Definition}

\newtheorem{lemma}[theorem]{Lemma}
\newtheorem{notation}[theorem]{Notation}

\begin{document}

\title{The Title of an IEEE Transactions Conference Article}

%

\title{Information Divergence is more $\chi^2$-distributed than the $\chi
^2$-statistics}
\author{\authorblockN{Peter Harremo{\"e}s}
\authorblockA{Copenhagen Business College\\
Copenhagen, Denmark\\
Email: harremoes@ieee.org}
\and\authorblockN{G{\'a}bor Tusn{\'a}dy}
\authorblockA{R\'enyi Institute of Mathematics\\
Budapest, Hungary\\
Email: tusnady@renyi.hu}}%
%

\maketitle
%

\begin{abstract}%

For testing goodness of fit it is very popular to use either the $\chi^{2}%
$-statistic or $G^{2}$-statistics (information divergence). Asymptotically
both are $\chi^{2}$-distributed so an obvious question is which of the two
statistics that has a distribution that is closest to the $\chi^{2}%
$-distribution. Surprisingly, when there is only one degree of freedom it
seems like the distribution of information divergence is much better
approximated by a $\chi^{2}$-distribution than the $\chi^{2}$-statistic. For
random variables we introduce a new transformation that transform several
important distributions into new random variables that are almost Gaussian.
For the binomial distributions and the Poisson distributions we formulate a
general conjecture about how close their transform are to the Gaussian. The
conjecture is proved for Poisson distributions.%

\end{abstract}%

\section{Choice of statistic}

We consider the problem of testing goodness-of-fit in a discrete setting. Here
we shall follow the classic approach to this problem as developed by Pearson,
Neyman and Fisher. The question is whether a sample with observation counts
$\left(  X_{1},X_{2},\dots,X_{k}\right)  $ has been generated by the
distribution $Q=\left(  q_{1},q_{2},\dots,q_{k}\right)  .$ For sample size $n$
the counts $\left(  X_{1},X_{2},\dots,X_{k}\right)  $ is assumed to have a
multinomial distribution. We introduce the empirical distribution $\hat
{P}=\left(  \frac{X_{1}}{n},\frac{X_{2}}{n},\dots,\frac{X_{k}}{n}\right)  $
where $n$ denotes the sample size $n=X_{1}+X_{2}+\dots+X_{k}.$ Often one uses
one of the Csisz\'{a}r \cite{Csiszar1963} $f$-divergences
\begin{equation}
D_{f}\left(  \hat{P},Q\right)  =%
{\textstyle\sum\limits_{j=1}^{k}}
q_{j}f\left(  \frac{\hat{p}_{j}}{q_{j}}\right)  . \label{Csi}%
\end{equation}
The null hypothesis is accepted if the test statistic $D_{f}\left(  \hat
{P},Q\right)  $ is small and rejected if $D_{f}\left(  \hat{P},Q\right)  $ is
large. Whether $D_{f}\left(  \hat{P},Q\right)  $ is considered to be small or
large depends on the significance level \cite{Lehman2005}. The most important
cases are obtained for the convex functions $f(t)=n(t-1)^{2}$ and $f(t)=2nt\ln
t$\ leading to the \emph{Pearson }$\chi^{2}$\emph{-statistic}%
\begin{equation}
\chi^{2}=%
{\textstyle\sum\limits_{j=1}^{k}}
\frac{(X_{nj}-nq_{nj})^{2}}{nq_{nj}} \label{Pears}%
\end{equation}
or the \emph{likelihood ratio statistic }%
\begin{equation}
G^{2}=2%
{\textstyle\sum\limits_{j=1}^{k}}
X_{nj}\ln\frac{X_{nj}}{nq_{nj}} \label{lik}%
\end{equation}
which is a scaled version of \emph{information divergence} that we will denote
$D$ without subscript. In this paper we shall focus on the case where there
are only two bins because this allow us to formulate in a qualitattive manner
in terms of what we will call the intersection conjecture. With only two bins
the multinomial distribution of counts can be described by a binomial
distribution. We will also consider the limiting case where the binomial
distribution is replaced by a Poisson distribution. This corresponds in a
sense to having only one bin.

\begin{notation}
Please note that we follow the notation from \cite{Everitt1998} by denoting
the likelihood ratio statistic by $G^{2}$ rather than $G$ as done in many
textbooks and articles. Our $G^{2}$ should not be confused with Getis--Ord's
$G$ statistic \cite{Zhang2008}.
\end{notation}

One way of choosing between various statistics is by computing their
asymptotic efficiency. In 1985 it was proved that the $G^{2}$-statistic\ is
more efficient in the Bahadur sense than the $\chi^{2}$-statistic, and this
result has been extended in a number of papers \cite{Quine1985, Beirlant2001,
Gyorfi2000, Harremoes2008, Harremoes2008e}. The asymptotic Bahadur efficiency
of $G^{2}$ implies that a much smaller sample size is needed when using
$G^{2}$ than when using $\chi^{2}$ if a fixed power should be achieved at a
very small significance level for some alternative. Since this type of result
only holds asymptotically for large sample sizes it may be difficult to use
for a specific finite sample size. Therefore we will turn our attention to
another important property for the choice of statistic.

\input{divergenceplot.TpX}\input{divergenceplot2.TpX}For the practical use of
a statistic it is important to calculate or estimate the distribution of the
statistic. This can be done by exact calculation, by approximations, or by
simulations. Exact calculations may be both time consuming and difficult.
Simulation often requires statistical insight and programming skills.
Therefore most statistical tests use approximations to calculate the
distribution of the statistic. For a fixed number of bins the distribution of
the $\chi^{2}$-statistic becomes closer and closer to the $\chi^{2}%
$-distributions as the sample size tends to infinity. For a large sample size
the empirical distribution will with high probability be close to the
generating distribution and the Csisz\'{a}r $f$-divergence $D_{f}$ can be
approximated by a scaled version of the $\chi^{2}$-statistic%
\[
D_{f}\left(  P,Q\right)  \approx\frac{f^{\prime\prime}\left(  0\right)  }%
{2}\cdot\chi^{2}\left(  P,Q\right)  .
\]
Therefore the distribution of any $f$-divergence may be approximated by a
scaled $\chi^{2}$-distribution, i.e. a $\Gamma$-distribution. From this
argument one might get the impression that the distribution of the $\chi^{2}%
$-statistic is closer to the $\chi^{2}$-distribution. Figure \ref{Fig1} and
Figure \ref{Fig2} show that this is far from the the case. Both figures are
Q-Q plots where for each $p\in\left[  0,1\right]  $ a point is plottet with
the $p$ quantile the square of a standard Gaussian as first coordinate and the
$p$ quantile of the distribution of the statistic as the second coordinate.
Figure \ref{Fig1} shows that the $G^{2}$-statistic is almost as $\chi^{2}%
$-distributed as it can be when one takes into account that the likelihood
ratio statistic has a discrete distribution. Each step is intersected very
close to its midpoint. Figure \ref{Fig2} shows that the distribution of the
$\chi^{2}$-statistic deviates systematically from the $\chi^{2}$-distribution
for small significance levels. For larger significance levels both statistics
will give approximately the same results which is related to the fact that the
two statistics have the same asymptotic Pitman efficiency. The two plots show
that at least in some cases the distribution of the $G^{2}$-statistic is much
closer to a $\chi^{2}$-distribution than Pearson statistic is. The next
question is whether there are situations where the likelihood ratio statistic
is not approximately $\chi^{2}$-distributed. For binomial distributions that
are very skewed the intersection property of Figure \ref{Fig1} is not
satisfied when the $G^{2}$-statistic is plotted against the $\chi^{2}%
$-distribution so in the rest of this paper a different type of plots will be
used. For getting a better approximation another strategy is Bartlett's
adjustment, see \cite{Barndorff-Nielsen1988}.

T. Dunning \cite{Dunning1993} has made a summary of what the typical
recommendations are about whether one should use the $\chi^{2}$-statistic or
the $G^{2}$-Statistic. The short version is that the statistic is
approximately $\chi^{2}$-distributed when each bin contains at least 5
observations or the calculated variance for each bin is at least 5, and if any
bin contains more than twice the expected number observations then the $G^{2}%
$-statistic is preferable to the $\chi^{2}$-statistic. Our main idea is to
\emph{change} the statistic into a signed version as it was introduced by
Barndorff-Nielsen as a signed likelihood ratio \cite{Barndorff-Nielsen1994}.
We call the operation $G$-transform and change our orientation from hypothesis
testing to normal approximation of distributions of sums of independent
variables. Our main observation is that the $G$-transform covers probabilities
in the whole domain including large deviations.

\begin{notation}
In the rest of this paper we will let $\tau$ denote the circle constant $2\pi$
and let $\phi$ denote the standard Gaussian density%
\[
\frac{\exp\left(  -\frac{z^{2}}{2}\right)  }{\tau^{1/2}}.
\]
We let $\Phi$ denote the distribution function of the standard Gaussian%
\[
\Phi\left(  t\right)  =\int_{-\infty}^{t}\phi\left(  z\right)
~\text{\textrm{d}}z~.
\]

\end{notation}

\section{The $G$-transform and its distribution}

Here we shall introduce a transformation that is useful for our understanding
of the fine structure of the distribution of the likelihood ratio statistics.
Consider a 1-dimensional exponential family $P_{\beta}$ where%
\[
\frac{\text{\textrm{d}}P_{\beta}}{\text{\textrm{d}}P_{0}}\left(  x\right)
=\frac{\exp\left(  \beta\cdot x\right)  }{Z\left(  \beta\right)  }%
\]
and $Z$ denotes the partition function given by%
\[
Z\left(  \beta\right)  =\int\exp\left(  \beta\cdot x\right)  ~\text{\textrm{d}%
}P_{0}x~.
\]
Let $P^{\mu}$ denote the element in the exponential family with mean value
$\mu.$ Let $\mu_{0}$ denote the mean value of $P_{0}.$ Then%
\[
D\left(  P^{\mu}\Vert P_{0}\right)  =\int\ln\left(  \frac{\text{\textrm{d}%
}P^{\mu}}{\text{\textrm{d}}P_{0}}\left(  x\right)  \right)  ~\text{\textrm{d}%
}P^{\mu}x.
\]
To verify that $D\left(  P^{\mu}\Vert P_{0}\right)  $ is $\chi^{2}%
$-distributed it is sufficient to verify that the square root is a centered
Gaussian. This motivates the next definition:

\begin{definition}
Let $X$ be a random variable with distribution $P_{0}.$ Then the
$G$\emph{-transform} $G\left(  X\right)  $ of $X$ is the random variable given
by%
\[
G\left(  x\right)  =\left\{
\begin{array}
[c]{cc}%
-\left(  2D\left(  P^{x}\Vert P_{0}\right)  \right)  ^{1/2}, & \text{for
}x<\mu_{0};\\
\left(  2D\left(  P^{x}\Vert P_{0}\right)  \right)  ^{1/2}, & \text{for }%
x\geq\mu_{0}.
\end{array}
\right.
\]

\end{definition}

Using $G\left(  x\right)  $ instead of $D\left(  P^{x}\Vert P_{0}\right)  $ as
statistic is essentially the difference between using a one-sided test instead
of a two-sided test. With this definition one easily sees that the
$G$-transform of a Gaussian is a standard Gaussian. In \cite{Gyorfi2012} it
was verified that if a random variable $X$ satisfies a Cram\'{e}r condition
the then with a minor correction $G_{n}\left(  \frac{1}{n}\sum_{i=1}^{n}%
X_{i}\right)  $ is Gaussian within a factor of the order $1+O\left(  \frac
{1}{n}\right)  .$ In this paper we are interested in sharp bounds rather than
asymptotic results.

Now we can make quantile plots of the Gaussian distribution against the
distribution of the $G$-transform of various random variables. On Figure 3-7
the $G$-transform of some binomial and Poisson distributions are compared with
the standard Gaussian via their Q-Q plot. In these plots the red lines
correspond to the bounds $P\left(  X\leq x\right)  \leq\exp\left(  -D\left(
P^{x}\Vert P_{0}\right)  \right)  $ for $x\leq\mu_{0}$ and $P\left(  X\geq
x\right)  \leq\exp\left(  -D\left(  P^{x}\Vert P_{0}\right)  \right)  $ for
$x\geq\mu_{0}$.\input{bin0,5.TpX} \input{bin0,3.TpX} \input{bin0,1.TpX} \input{lamdaLig20.TpX}

These plots support the following conjecture:

\begin{conjecture}
[The intersection property]Let $M$ denote a binomial distributed or Poisson
distributed random variable and let $G\left(  M\right)  $ denote the
$G$-transform of $M.$ The quantile transform between $G\left(  M\right)  $ and
a standard Gaussian $Z$ is a step function and the identity function
intersects each step, i.e.
\[
P\left(  M<m\right)  <P\left(  Z\leq G\left(  m\right)  \right)  <P\left(
M\leq m\right)
\]
for all integers $m.$
\end{conjecture}

Another way of reformulating the intersection property is that in the
stochastic ordering $X$ should be less than a random variable with point
probabilities $\Phi\left(  G\left(  m\right)  \right)  -\Phi\left(  G\left(
m-1\right)  \right)  $ and greater than a random variable with point
probabilities $\Phi\left(  G\left(  m+1\right)  \right)  -\Phi\left(  G\left(
m\right)  \right)  ,$ where $G\left(  -1\right)  $ is defined as $-\infty$ and
$G\left(  n+1\right)  $ is defined to be $\infty$ for a binomial distribution
number parameter $n.$ The conjecture is so well supported by numerical
calculations that we would not hesitate to recommend it for estimation of tail
probabilities for the binomial distributions in goodness of fit tests instead
of using the usual $\chi^{2}$-approximation of the $\chi^{2}$-statistic.

As we see both skewed binomial distributions and Poisson distributions have
different step sizes for positive and negative values. Although the quantile
transform between $G\left(  M\right)  $ and a standard Gaussian has the
intersection property interference between the step sizes may have the effect
that the quantile transform between the $G^{2}$-statistic and the $\chi^{2}%
$-distribution does not necessarily have the intersection property. We believe
that the $G$-transform is always closer to a standard Gaussian than the
original. We have no idea, which distributions have the intersection property.

\section{The link to waiting times}

Hitherto we have discussed inequalities for discrete distributions but there
is an interesting link to inequalities for continuous distributions associated
with waiting times. Assume that $M$ is Poisson distributed with mean $t$ and
$T$ is Gamma distributed with shape parameter $m$ and scale parameter 1, i.e.
the distribution of the waiting time until $m$ observations from an Poisson
process with intensity 1. Then%
\begin{equation}
P\left(  M\geq m\right)  =P\left(  T\leq t\right)  .
\label{Eq:PoissonGammaLink}%
\end{equation}
The Gamma distribution $\Gamma\left(  m,\theta\right)  $ has density%

\[
f\left(  t\right)  =\frac{1}{\theta^{m}}\frac{1}{\Gamma\left(  m\right)
}t^{m-1}\exp\left(  -\frac{t}{\theta}\right)
\]
so the divergence can be calculated as%
\[
D\left(  \Gamma\left(  m,\theta_{1}\right)  \Vert\Gamma\left(  m,\theta
_{2}\right)  \right)  =m\left(  \frac{\theta_{1}}{\theta_{2}}-1-\ln
\frac{\theta_{1}}{\theta_{2}}\right)  .
\]
In particular%
\[
D\left(  \Gamma\left(  m,\frac{t}{m}\right)  \Vert\Gamma\left(  m,1\right)
\right)  =t-m-m\ln\frac{t}{m}.
\]
Next we note that
\[
D\left(  Po\left(  m\right)  \Vert Po\left(  t\right)  \right)  =D\left(
\Gamma\left(  m,\frac{t}{m}\right)  \Vert\Gamma\left(  m,1\right)  \right)  .
\]
If $G_{P}$ is the $G$-transform for $Po\left(  t\right)  $ and $G_{\Gamma}$ is
the $G$-transform for $\Gamma\left(  m,1\right)  $ then $G_{P}\left(
m\right)  =-G_{\Gamma}\left(  t\right)  .$ This shows that if the
$G$-transforms of the Gamma distributions are close to a Gaussian then so are
the $G$-transforms of the Poisson distributions. Figure \ref{Gammaplot} shows
Q-Q plots of the $G$-transform of some Gamma distributions.\input{Gammaplot.TpX}

We see that the fit with a straight line of slope 1 is extremely good. The
point (0,0) is not on the line reflecting the fact that the means and the
medians of the Gamma distributions do not coincide. In the next section we
shall see that the quantile transform between a Gaussian and the $G$-transform
of Gamma distributions is always below the identity.

\section{The increasing property}

In this section we shall formulate some conditions that are stronger than the
intersection property. The proof of the following lemma is an easy exercise so
we omit the proof.

\begin{lemma}
\label{increacedensity}Let $f_{1}$ and $f_{2}$ be the densities of the random
variables $X_{1}$ and $X_{2}$ with respect to some measure $\mu$ on the real
numbers. If
\[
\frac{f_{1}}{f_{2}}%
\]
is an increasing then $X_{1}$ is less than $X_{2}$ in the usual stochastic ordering.
\end{lemma}

\begin{theorem}
\label{Thm:Gamma}The $G$-transform of a Gamma distributed random variable is
less than a standard Gaussian in the stochastic ordering.

\begin{proof}
Let $T$ be a $\Gamma\left(  m,1\right)  $ distributed random variable with
density $g$. The distribution in the exponential family based on
$\Gamma\left(  m,1\right)  $ with mean $t$ is $\Gamma\left(  m,\frac{t}%
{m}\right)  .$ The $G$-transform is
\[
G\left(  t\right)  =\pm\left(  2D\left(  \Gamma\left(  m,\frac{t}{m}\right)
\Vert\Gamma\left(  m,1\right)  \right)  \right)  ^{1/2}%
\]
where $\pm$ means that we will use $+$ when $t$ is greater than the mean $k$
and use $-$ when $t$ is less than $m.$ For the Gamma distribution we have
\[
\frac{\mathrm{d}\Gamma\left(  m,\frac{t}{m}\right)  }{\mathrm{d}\Gamma\left(
m,1\right)  }\left(  t\right)  =\frac{\exp\left(  t-m\right)  }{\left(
\frac{t}{m}\right)  ^{m}}.
\]
Let $W=G\left(  T\right)  $ with density $f\left(  w\right)  .$ We want to
prove that $\frac{\phi\left(  w\right)  }{f\left(  w\right)  }$ is increasing.
Now
\[
f\left(  w\right)  =\frac{g\left(  G^{-1}\left(  w\right)  \right)
}{G^{\prime}\left(  G^{-1}\left(  w\right)  \right)  }%
\]
so that%
\begin{align*}
\frac{\phi\left(  w\right)  }{f\left(  w\right)  } &  =\frac{\phi\left(
G\left(  t\right)  \right)  G^{\prime}\left(  t\right)  }{g\left(  t\right)
}=\frac{G^{\prime}\left(  t\right)  }{\tau^{1/2}\frac{\mathrm{d}\Gamma\left(
m,\frac{t}{m}\right)  }{\mathrm{d}\Gamma\left(  m,t\right)  }\left(  t\right)
\cdot g\left(  t\right)  }\\
&  =\frac{\Gamma\left(  m\right)  }{\tau^{1/2}m^{m}\exp\left(  -m\right)
}\cdot tG^{\prime}\left(  t\right)  .
\end{align*}
Hence we want to prove that $tG^{\prime}\left(  t\right)  $ is increasing.%
\[
tG^{\prime}\left(  t\right)  =\pm t\frac{2D^{\prime}\left(  t\right)
}{2\left(  2D\right)  ^{1/2}}=\pm m^{1/2}\frac{\frac{t}{m}-1}{\left(  2\left(
\frac{t}{m}-1-\ln\frac{t}{m}\right)  \right)  ^{1/2}}.
\]
With the substitution $u=t/m$ we have to prove that%
\[
\pm\frac{u-1}{\left(  2\left(  u-1-\ln u\right)  \right)  ^{1/2}}%
\]
is increasing. We have
\[
\frac{\mathrm{d}}{\mathrm{d}u}\left(  \pm\frac{u-1}{\left(  2\left(  u-1-\ln
u\right)  \right)  ^{1/2}}\right)  =\pm\frac{u-2\ln u-\frac{1}{u}}{\left(
2\left(  u-1-\ln u\right)  \right)  ^{3/2}}%
\]
so we want to prove that%
\[
\pm\left(  u-2\ln u-\frac{1}{u}\right)  \geq0.
\]
Now we have to prove that $\ell\left(  u\right)  =u-2\ln u-\frac{1}{u}$ is
positive for $u>1$ and negative for $u<1.$ Obviously $\ell\left(  1\right)
=0$ so it is sufficient to prove that $\ell^{\prime}\left(  u\right)  \geq0,$
but
\[
\ell^{\prime}\left(  u\right)  =\left(  1-\frac{1}{u}\right)  ^{2}\geq0.
\]

\end{proof}
\end{theorem}

Next we shall formulate an even stronger conjecture and see that it actually
implies that binomial distributions and Poisson distributions have the
intersection property.

\begin{conjecture}
[The increasing property]If $M$ is a binomially or Poisson distributed random
variable with $G$-transform $G\left(  M\right)  $ then
\begin{equation}
m\rightarrow\frac{P\left(  M=m\right)  }{\Phi\left(  G\left(  m+1\right)
\right)  -\Phi\left(  G\left(  m\right)  \right)  }.\label{Eq:increas}%
\end{equation}
is increasing and
\[
m\rightarrow\frac{P\left(  M=m\right)  }{\Phi\left(  G\left(  m\right)
\right)  -\Phi\left(  G\left(  m-1\right)  \right)  }%
\]
is decreasing.
\end{conjecture}

The conjecture is supported by numerous numerical computations. If it holds
the intersection property follows by Lemma \ref{increacedensity}. The
increasing property implies log-concavity of the distribution but for instance
the geometric distribution is log-concave but does not satisfy the
intersection property. We have some indications that the conjecture also holds
for any distribution of a sum of independent Bernoulli random variables.\input{increas.TpX}

\begin{theorem}
\label{Thm:Poisson}The intersection property is satisfied for any Poisson
random variable.
\end{theorem}

\begin{proof}
(Outline) The inequality%
\[
P\left(  M<m\right)  \leq P\left(  Z\leq G\left(  m\right)  \right)
\]
follows from Theorem \ref{Thm:Gamma} combined with Equation
\ref{Eq:PoissonGammaLink}. The inequality%
\[
P\left(  M\leq m\right)  \geq P\left(  Z\leq G\left(  m\right)  \right)
\]
can be proved case by case for $m\leq5.$ For $m>5$ it is proved using the
intersection property$.$
\end{proof}

Theorem (\ref{Thm:Poisson}) gives bounds on the tail probabilities for Poisson
distributions that are far better than what can be found in the literature
(see for instance \cite{Glynn1987}). At the same time the theorem gives bounds
on the median that are compatible with the bounds in the literature
\cite{Chen1986, Hamza1995}.

\section{Discussion}

Many goodness-of -fit tests involve parameter estimation (that is, the model
is a parametric family of distributions, not a single distribution). In such
cases, the $G^{2}$-statistic may converge slower to a $\chi^{2}$-distribution
than the $\chi^{2}$-statistic \cite{Perkins2011}. How such results are related
to the presents results is now clear yet. Since we only discuss the cases with
one or two bins our results can be reformulated in terms of conficence
intervals. We hope to cover confidence intervals in a future paper.

In the present paper the focus has been on the two bin case. We do not know if
something equivalent of the intersection property can be formulated for more
than two bins. For results on more than two bins it may be better to try to
generalize the results on asymptotics presented in \cite{Gyorfi2012}.

\section{Acknowledgement}

The authors want to thank Unnikrishnan Jayakrishnan for useful discussions. We
also want to thank Jen\H{o} Reiczigel and L\'{\i}dia Rejt\H{o} for helping
with some numerical computations at an early stage of the developing the ideas
presented in this paper and L\'{a}szl\'{o} Gy\"{o}rfi who we worked in
parallel on other aspects of the intersection conjecture. Finally we would
like to thank Sune Jakobsen for comments to this manuscript.

\bibliographystyle{ieeetr}
\bibliography{database1}%

\pagebreak
\end{document}